\documentclass{amsart}
\usepackage{graphicx}
\usepackage{amscd}
\usepackage{color}
\newtheorem{theorem}{Theorem}[section]
\newtheorem{lemma}[theorem]{Lemma}

\newtheorem{example}[theorem]{Example}
\newtheorem{corollary}[theorem]{Corollary}
\newtheorem{conjecture}[theorem]{Conjecture}

\theoremstyle{definition}

\newtheorem{remark}[theorem]{Remark}

\begin{document}

\title{Height, trunk and representativity of knots}

\author{Ryan Blair}
\address{Department of Mathematics, California State University, Long Beach, 1250 Bellflower Blvd, Long Beach, CA 90840}
\email{ryan.blair@csulb.edu}

\author{Makoto Ozawa}
\address{Department of Natural Sciences, Faculty of Arts and Sciences, Komazawa University, 1-23-1 Komazawa, Setagaya-ku, Tokyo, 154-8525, Japan}
\email{w3c@komazawa-u.ac.jp}
\thanks{The author is partially supported by Grant-in-Aid for Scientific Research (C) (No. 26400097 \& 17K05262), The Ministry of Education, Culture, Sports, Science and Technology, Japan}



\begin{abstract}
In this paper, we investigate three geometrical invariants of knots, the height, the trunk and the representativity.

First, we give a conterexample for the conjecture which states that the height is additive under connected sum of knots.
We also define the minimal height of a knot and give a potential example which has a gap between the height and the minimal height.

Next, we show that the representativity is bounded above by a half of the trunk.
We also define the trunk of a tangle and show that if a knot has an essential tangle decomposition, then the representativity is bounded above by half of the trunk of either of the two tangles.

Finally, we remark on the difference among Gabai's thin position, ordered thin position and minimal critical position.
We also give an example of a knot which bounds an essential non-orientable spanning surface, but has arbitrarily large representativity.
\end{abstract}

\maketitle

%\tableofcontents

\section{Introduction}

We study a knot in the 3-sphere via a standard Morse function $h: S^3 \to \mathbb{R}$.
We derive two geometrical invariants of a knot, one is  ``height" from the vertical direction of $h$, and another is ``trunk" from the horizontal direction.

\subsection{Height of knots}\label{sub1}
It is often difficult to determine how geometrically defined knot invariants behave with respect to connected sum. Some classical invariants are known to be predictably well-behaved such as genus and bridge number \cite{S}. While others are only conjectured to be well-behaved such as crossing number and unknotting number. Still others have been shown to exhibit complicated behavior with respect to connected sum such as tunnel number \cite{Mori1} and width \cite{BT1}. In this paper we study the behavior of \emph{height} of a knot with respect to connected sum and show that this invariant best fits in the third category by demonstrating that height is not additive with respect to connected sum, giving a counterexample to Conjecture 3.5 of \cite{Ozawa1}.

Let $K$ be an ambient isotopy class of knot in $S^3$ and let $h: S^3 \to \mathbb{R}$ be the standard height function. If $\gamma$ is a smooth embedding of knot type $K$, $h|_{\gamma}$ is morse and all critical points of $h|_{\gamma}$ have distinct critical values, then we will write $\gamma\in K$. Though an abuse of notation, we will also let $\gamma$ denoted the image of the embedding.

Then the {\em bridge number} of $\gamma\in K$ is denoted $\beta(\gamma)$ and is defined to be the number of maxima of $h|_{\gamma}$. The {\em bridge number} of $K$, $\beta(K)$, is the minimum of $\beta(\gamma)$ over all $\gamma \in K$. Schubert showed that the bridge number of a connected sum $K_1\# K_2$ always satieties the equality
\[
\beta(K_1\# K_2)=\beta(K_1)+\beta(K_2)-1.
\]

Bridge number is closely related to the width of a knot which was originally defined by Gabai and used in the proof of the property R conjecture \cite{Gabai1}. To define width, we first need some additional structure. If $t$ is a regular value of $h|_{\gamma}$, then $h^{-1}(t)$ is called a {\em level
sphere} with width $w(h^{-1}(t))=|\gamma \cap h^{-1}(t)|$. If $c_{0}<c_{1}<...<c_{n}$ are all the critical values of
$h|_{\gamma}$, choose regular values $r_{1},r_{2},...,r_{n}$ such that $c_{i-1}<r_{i}<c_{i}$. Then the {\em width of $\gamma$} is defined by $w(\gamma)=\sum w(h^{-1}(r_{i}))$. The {\em width} of $K$, $w(K)$, is the minimum
of $w(\gamma)$ over all $\gamma \in K$. We say that $\gamma \in K$ is a {\em thin position} for $K$ if $w(\gamma)=w(K)$ and write $\gamma\in TP(K)$ where $TP(K)$ denotes the set of all thin positions of $K$.

Based in part on Schubert's equality, it was widely conjectured that the width of a connected sum $K_1\# K_2$ always satieties the equality
\[
w(K_1\# K_2)=w(K_1)+w(K_2)-2.
\]
Rieck and Sedgwick made progress on this conjecture when they showed that the above equality always holds when $K_1$ and $K_2$ are mp-small knots \cite{RS}. Additionally, Scharlemann and Schultens showed that $w(K_1\# K_2)\geq \max\{w(K_1),w(K_2)\}$ \cite{SchSch}. However, Scharlemann and Thompson proposed counterexamples to the equality in \cite{SchTh} and Blair and Tomova proved that an infinite class of the Scharlemann-Thompson examples were counterexamples \cite{BT1}. However, there are alternative definitions of width for which width is well-behaved with respect to connected sum \cite{TT1}.  In general, the best known inequalities for $w(K_1\# K_2)$ are
\[
\max\{w(K_1),w(K_2)\}\leq w(K_1\# K_2) \leq w(K_1)+w(K_2)-2
\]
with each of these individual inequalities known to be equalities for certain choices of $K_1$ and $K_2$.

To define the height of a knot, we first need to introduce the notion of thick and thin level. A level sphere $h^{-1}(t)$ for $\gamma\in K$ is called {\em thin} if the highest critical point for $\gamma$ below it is a maximum and
the lowest critical point above it is a minimum. If the highest critical point for $\gamma$ below $h^{-1}(t)$ is a
minimum and the lowest critical point above it is a maximum, the level sphere is called {\em thick}. As the lowest
critical point of $K$ is a minimum and the highest is a maximum, a thick level sphere can always be found. Note that some embeddings will have no thin spheres. When this occurs the unique thick sphere is called a {\em bridge sphere} and the embedding is said to be a {\em bridge position} for $K$.

Given $\gamma \in K$ such that $\gamma$ is a thin position, we define the {\em height} of $\gamma$, denoted $ht(\gamma)$ to be the number of thick level spheres for $\gamma$. Similarly, the \emph{height} of a knot type $K$ is defined as
\[
\displaystyle ht(K)=\max_{\gamma\in TP(K)} ht(\gamma).
\]

Alternatively, we will define the \emph{min-height} of a knot type $K$ to be

\[
\displaystyle ht_{\min}(K)=\min_{\gamma\in TP(K)} ht(\gamma).
\]

Clearly, $ht_{\min}(K)\leq ht(K)$ for all knots $K$.

The basic properties of height have been studied in \cite{Ozawa2}.
Here, we would like to point out the relation between the height and essential planar surfaces.
Bachman proved in \cite{B} that if a knot or link has $n$ thin levels when put in thin position then its exterior contains a collection of $n$ disjoint, non-parallel, planar, meridional, essential surfaces.
Let $mp(K)$ denote the maximal number of disjoint, non-parallel, planar, meridional, essential surfaces in the exterior of $K$.
Then, we obtain the following inequality.

\begin{theorem}[\cite{B}, cf. \cite{HK1997}]\label{htep}
\[
ht(K)\le mp(K)+1
\]
\end{theorem}

\begin{remark}
The inequality in Theorem \ref{htep} can take an arbitrarily gap generally.
For example, a knot $K$ in the left of Figure \ref{counterexample3} has $ht(K)=1$, but $mp(K)\ge 3$.
\end{remark}

We are interested in understanding how height behaves with respect to connected sum.
On the other hand, it holds that $mp(K_1\# K_2)\ge mp(K_1)+mp(K_2)+1$ and this inequality can take an arbitrarily gap generally.
Indeed let $K_i$ be a knot with $mp(K_i)=0$ for $i=1,2,3$.
Then $mp(K_1\# K_2)=mp(K_1)+mp(K_2)+1=1$ and 
$mp((K_1\# K_2)\# K_3)=3>mp(K_1\# K_2)+mp(K_3)+1=2$.

It was remarked in \cite{Ozawa1} that the height is additive with respect to connected sum for meridionally small knots (cf. \cite[Theorem 1.8]{O2}), and conjectured that for non-trivial knots $K_1$ and $K_2$, $ht(K_1\# K_2) = ht(K_1)+ht(K_2)$ always holds. By a similar argument, it follows that min-height is also additive with respect to connected sum for meridionally small knots. Hence, it is natural to ask if min-height is always additive with respect to connected sum. Our first results provide counterexamples to each of these conjectures by defining an infinite class of knots $\mathcal{K}$ such that for every $K\in \mathcal{K},$ $ht_{\min}(K)=ht(K)=3$ and the following Theorem holds.

\begin{theorem}\label{main}
Let $K\in \mathcal{K}$ and let $K_2$ be any two-bridge knot, then
$$
ht(K\# K_2)=ht(K)=3,
$$
$$
ht_{\min}(K\# K_2)=ht_{\min}(K)=3.
$$
\end{theorem}

Since the height of any two-bridge knot is one, Theorem \ref{main} gives a counterexample to Conjecture 3.5 of \cite{Ozawa1} that for all knots $K_1$ and $K_2$, $ht(K_1 \# K_2)=ht(K_1)+ht(K_2)$.
By Theorem \ref{main}, the additivity of height does not hold with respect to connected sum of knots.
At this stage, we expect the following.

\begin{conjecture}\label{height}
For any two knots $K_1, K_2$, it holds that 
\[
\max\{ht(K_1), ht(K_2) \} \le ht(K_1\# K_2) \le ht(K_1) + ht(K_2),
\]
\[
\max\{ht_{\min}(K_1), ht_{\min}(K_2) \} \le ht_{\min}(K_1\# K_2) \le ht_{\min}(K_1) + ht_{\min}(K_2).
\]
\end{conjecture}

For a knot $K$ given in Theorem \ref{main}, we have $ht_{\min}(K)=ht(K)$.
But it is natural to think that the gap between the min-height and the height can be taken arbitrarily large in general.

\begin{conjecture}\label{gap}
There exists a knot $K$ such that $ht_{\min}(K)<ht(K)$.
\end{conjecture}

We give a potential counterexample $K$ for Conjecture \ref{gap} in Figure \ref{counterexample4}.
The width of the embedding $\gamma\in K$ on the left is $\frac{1}{2}(22)^2=242$ while the width of the embedding on the right $\gamma'\in K$ is $\frac{1}{2}(18)^2+\frac{1}{2}(14)^2-\frac{1}{2}(6)^2=242$.

\begin{figure}[htbp]
	\begin{center}
	\begin{tabular}{cc}
	\includegraphics[trim=0mm 0mm 0mm 0mm, width=.35\linewidth]{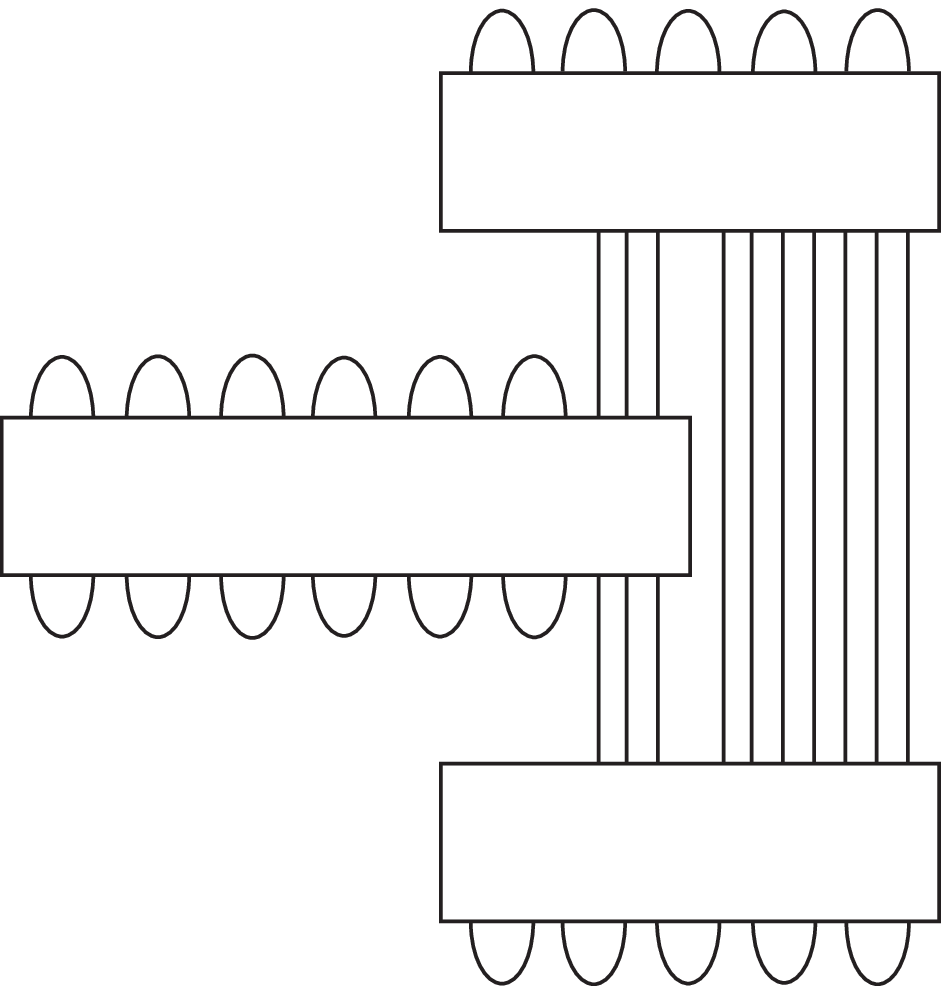}&
	\includegraphics[trim=0mm 0mm 0mm 0mm, width=.45\linewidth]{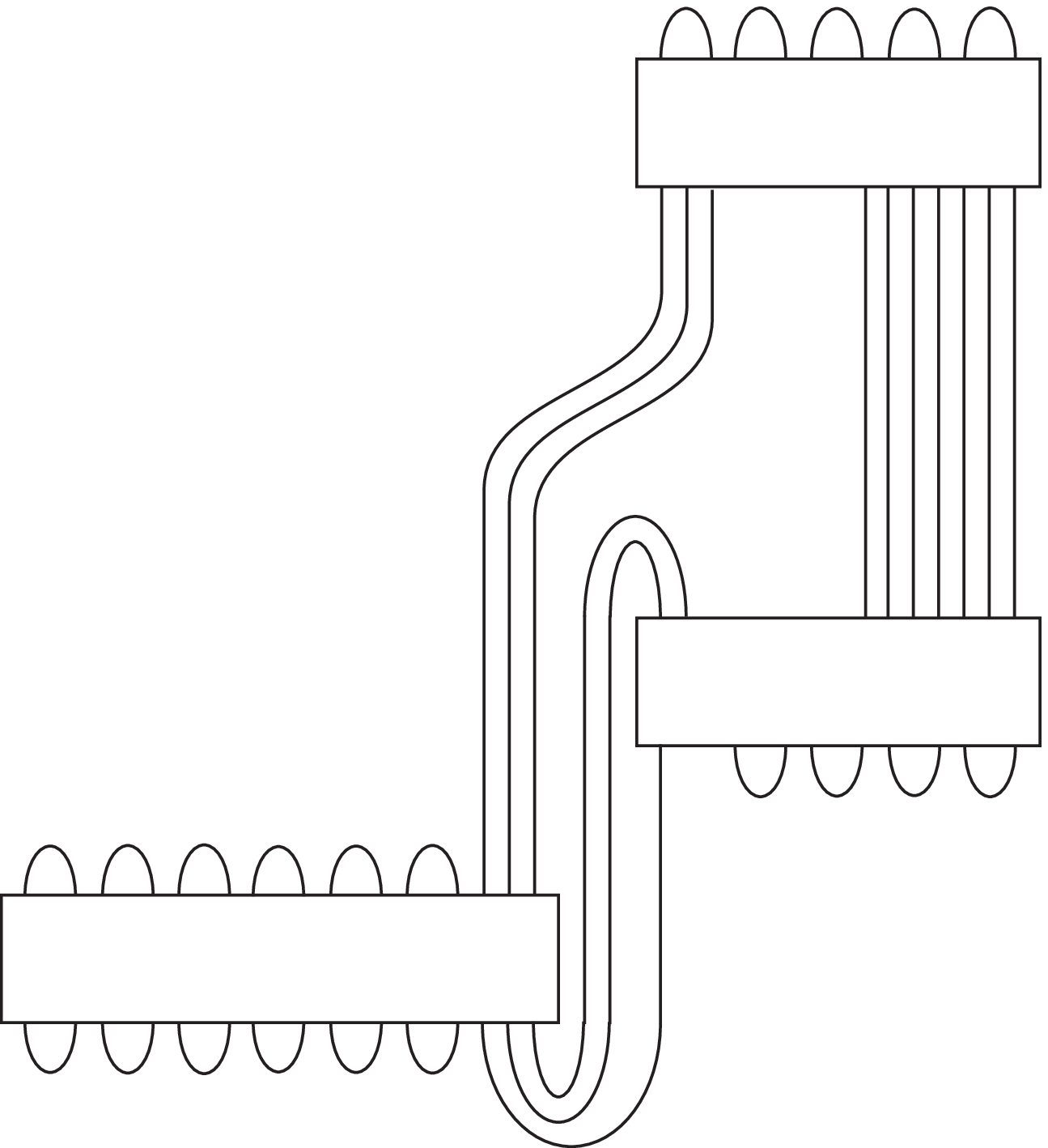}\\
	$\gamma$ & $\gamma'$
	\end{tabular}
	\end{center}
	\caption{A potential counterexample for Conjecture \ref{gap}, where $w(\gamma)=w(\gamma')=242$.}
	\label{counterexample4}
\end{figure}

\subsection{Representativity and trunk of knots}\label{sub2}
To measure the ``density" of a graph embedded in a closed surface, Robertson and Vitray introduced the {\em representativity} in \cite{RV} as the minimal number of points of intersection between the graph and any essential closed curve on the closed surface.
This concept was applied to a knot $K$ in the 3-sphere $S^3$ in \cite{O1} and extended to a spatial graph in the 3-sphere in \cite{O3}.
Let $F$ be a closed surface containing the knot $K$.
We define the {\em representativity} of a pair $(F,K)$ as
\[
\displaystyle r(F,K)=\min_{D\in\mathcal{D}_F} |\partial D\cap K|,
\]
where $\mathcal{D}_F$ denotes the set of all compressing disks for $F$.
Moreover, we define the {\em representativity} of a knot $K$ as
\[
r(K)=\max_{F\in\mathcal{F}} r(F,K),
\]
where $\mathcal{F}$ denotes the set of all closed surfaces containing $K$.

The representativity measures the ``spatial density" of a knot.
We summarize the known values of representativity of knots.
\begin{enumerate}
\item $r(K)=1$ if and only if $K$ is the trivial knot (\cite[Example 3.2]{O3}).
\item $r(K)=2$ for composite knots (\cite[Example 3.6]{O3})
\item $r(K)\le 2n$ for knots with essential $n$-string tangle decompositions (\cite[Example 3.6]{O3})
\item $r(K)=2$ for 2-bridge knots (\cite[Example 3.4]{O3})
\item $r(K) = \min\{p,q\}$ for $(p,q)$-torus knots (\cite[Example 3.3]{O3})
\item $r(K)\le 3$ for algebraic knots (\cite[Theorem 1.5]{O4} for large case. The small case follows \cite[Theorem 1.2]{O3} since small algebraic knots are Montesinos knots with length 3.)
\item For a $(p,q,r)$-pretzel knot $K$, $r(K)=3$ if and only if $(p, q, r) = \pm(-2, 3, 3)$ or $\pm(-2, 3, 5)$ (\cite{O4}).
\item $r(K)=2$ for alternating knots (\cite{K})
\item $r(K)=p$ for inconsistent cable knots with index $p$ (\cite{AKT})
\item $r(K)\le \beta(K)$ ({\cite[Theorem 1.2]{O3}}).
\item $r(K)\le 160 \delta(K)$, where $\delta(K)$ denotes the distortion of $K$ (\cite{P}).
\end{enumerate}

We remark that the inequality (10) was used to show the above (1), (4), (5), (6), (7).
In this paper, we refine the inequality (10).

As in \cite{O2}, we define the {\em trunk} of a knot $K$ as
\[
trunk(K)=\min_{\gamma\in K} \max_{t\in \mathbb{R}} |h^{-1}(t)\cap \gamma|,
\]
It follows by the definition that $trunk(K)\le 2\beta(K)$.

The bridge number of knots behaves as expected under taking connected sums, that is, Schubert proved that  $\beta(K_1\#K_2)=\beta(K_1)+\beta(K_2)-1$ (\cite{S}).
On the other hand, it was naturally expected that  $trunk(K_1\# K_2)=\max\{trunk(K_1),$ $trunk(K_2)\}$ (\cite[Conjecture 1.7]{O2}).
Davies and Zupan showed in \cite{DZ} that this is true, namely,
for two knots $K_1$ and $K_2$,
\[
trunk(K_1\# K_2)=\max\{trunk(K_1),\ trunk(K_2) \}
\]

In several cases, the trunk turned out to be useful.
For example, it was shown in \cite{T} that $m(K)\ge trunk(K)/2$, where $m(K)$ denotes the multiplicity index of $K$.
It was also shown in \cite{IPSSVAS} that a knot $K$ is embeddable into $(m\times n)$-tube if and only if $trunk(K)<(m+1)(n+1)$.

The following theorem refines {\cite[Theorem 1.2]{O3}}.

\begin{theorem}\label{rt}
For any knot $K$, we have
\[
\displaystyle r(K)\le \frac{trunk(K)}{2}.
\]
\end{theorem}

In the following, we introduce a ``local trunk" of a knot, that is, the trunk of a tangle which lies in the pair of the 3-sphere and a knot.

Let $(B,T)$ be a tangle, where $B$ is a 3-ball and $T$ is proper ambient isotopy class of properly embedded arcs in $B$.
Let $h:B\to \mathbb{R}$ be a standard Morse function with a single maximal point $p$ and $B-p\cong \partial B\times (1,0]$.

If $\gamma$ is a smooth embedding in the proper ambient isotopy class of $T$, $h|_{\gamma}$ is morse and all critical points of $h|_{\gamma}$ in the interior of $\gamma$ have distinct critical values, then we will write $\gamma\in T$.

We define the {\em trunk} of a tangle $(B,T)$ as
\[
trunk(B,T)=\min_{\gamma\in T} \max_{t\in \mathbb{R}} |h^{-1}(t)\cap \gamma|.
\]

Then we obtain the next theorem which is a local version of Theorem \ref{rt}.

\begin{theorem}\label{lrt}
Let $K$ be a knot admitting an essential tangle decomposition $(S^3,K)=(B_1,T_1)\cup(B_2,T_2)$.
Then we have
\[
\displaystyle r(K)\le \frac{\min\{ trunk(B_1,T_1), trunk(B_2,T_2) \}}{2}.
\]
\end{theorem}

In some cases, Theorem \ref{lrt} is more useful than Theorem \ref{rt} and (3) in Subsection \ref{sub2}.
Indeed, we can reprove (6) above after Theorem \ref{lrt}.

\begin{corollary}[{\cite[Theorem 1.5]{O4}}]
For a large algebraic knot $K$, $r(K)\le 3$.
\end{corollary}

\begin{proof}
Let $K$ be a large algebraic knot (i.e. algebraic knot with an essential Conway sphere).
Then, $K$ admits an essential tangle decomposition $(S^3,K)=(B_1,T_1)\cup(B_2,T_2)$, where $(B_1,T_1)$ is a union of two rational tangles.
It is easy to see that $trunk(B_1,T_1)=6$.
By Theorem \ref{lrt}, we obtain $r(K)\le 3$.
\end{proof}

Theorem \ref{lrt} can be regarded as a local version of Theorem \ref{rt}, namely, a local property determines a global property.
Such results can be seen in Theorem 3.1 in \cite{BZ} which restates Theorem 4.4 in \cite{JT} for the bridge number, and in \cite{K2} and \cite{R} for determinants.

%%%%%%%%%%%%%%%%%%%%%%%%%%%%%%%%%%%%%%%%%%%%%%%%%%

\section{Proof of theorems}

\subsection{Proof of Theorem \ref{main}}
In this subsection we utilize the results in \cite{SchSch} to give a lower bound on the height and min-height of some satellite knots.

The following theorem is Corollary 5.4 in \cite{SchSch}.

\begin{theorem}\label{reimbed}
Suppose $h: S^3 \to \mathbb{R}$ is the standard height function and $H\subset S^3$ is a handlebody for which horizontal circles of $\partial H$ with respect to $h$ constitute a complete collection of meridian disk boundaries. Then there is a reimbedding $f:H\rightarrow S^3$ so that
\begin{enumerate}
\item $h=h\circ f$ on $H$ and
\item $f(H)\cup (S^3\setminus f(H))$ is a Heegaard splitting of $S^3$.
\end{enumerate}
\end{theorem}

The proof of the following theorem is a slight variation on the proof of Corollary 6.3 of \cite{SchSch}.

\begin{theorem}\label{htineq}
Suppose $\gamma$ is an embedding of knot-type $K$ in an unknotted solid torus $H$ in $S^3$. Suppose $f:H\rightarrow S^3$ is a knotted embedding of $H$ and $\gamma'=f(\gamma)$ is an embedding of knot-type $K'$. If $w(K)=w(K')$, then $ht_{\min}(K')\geq ht_{\min}(K)$ and $ht(K')\leq ht(K)$.
\end{theorem}

\begin{proof}
Let $\gamma^*\in TP(K')$ such that $\gamma^*$ has $ht_{\min}(K')$ thick levels with respect to $h$. Let $H^*$ be the image of $f(H)$ under an isotopy taking $\gamma'$ to $\gamma^*$. We can additionally assume $\partial H^*$ is Morse with respect to $h$ after this isotopy. For every regular value $s$ of $h|_{\partial H^*}$, $(h|_{\partial H^*})^{-1}(s)$ is an unlink in $S^3$. By standard Morse theory and since $\partial H^*$ is a torus, there exists a regular value $s^*$ such that $(h|_{\partial H^*})^{-1}(s^*)$ has a component $c$ that is an essential loop in $\partial H^*$. Moreover, since $H^*$ is a knotted solid torus, $c$ is a meridian curve for $H^*$. By Theorem \ref{reimbed}, there is a reimbedding $g:H^*\rightarrow S^3$ of $H^*$ that preserves height and results in $g(H^*)$ being unknotted. Moreover, after a suitable choice of $g$, we can assume that $g(\gamma^*)\in K$, see \cite{SchSch} for details. Since $g$ is height preserving, $\gamma^*$ and $g(\gamma^*)$ have the same number of thick levels and $w(\gamma^*)=w(g(\gamma^*))$. Since $\gamma^*$ is a thin position for $K'$ and $w(K)=w(K')$, then $g(\gamma^*)$ is a thin position for $K$.  Hence, $ht_{\min}(K')\geq ht_{\min}(K)$.

Alternatively, let $\gamma^* \in TP(K')$ such that $\gamma^*$ has $ht(K')$ thick levels with respect to $h$. By the same argument as give above, we can find a height preserving reimbedding $g$ and $g(\gamma^*)\in K$ such that $w(K)=w(K')=w(\gamma^*)=w(g(\gamma^*))$ and $ht(K')=ht(\gamma^*)=ht(g(\gamma^*))$. Hence, $ht(K')\leq ht(K)$.
\end{proof}

\begin{remark}
It is interesting to note that Theorem \ref{htineq} does not hold if the hypothesis of $w(K)=w(K')$ is omitted. For example, if $L$ is a 2-bridge knot, it is an easy exercise to show that $ht(L)=1$ and $ht(L\# L)=2$. However, declaring $L=K$ and $L\# L=K'$ meets all the hypotheses of Theorem \ref{htineq} except $w(K)=w(K')$. Additionally, in Figure \ref{counterexample3} we give an example of a thin position $\gamma'$ for a knot-type $K'$ embedded in a knotted solid torus $f(H)$ together with an embedding $\gamma$ of knot type $K$ contained in the unknotted solid torus $H$ such that $\gamma'=f(\gamma)$. The embedding $\gamma'$ depicted in Figure \ref{counterexample3} is a thin position of $K'$ by Lemma 6.0.6 of \cite{BThesis}. Hence $ht_{\min}(K')=1$. Since $\beta(K)=4$, then the embedding $\gamma$ in the figure illustrates that no bridge position for $K$ is a thin position. Hence, $ht_{\min}(K)\geq2$ and moreover $ht_{\min}(K)=2$ since $mp(K)=1$ by \cite{O1984}. Thus, $1=ht(K')=ht_{\min}(K')< ht_{\min}(K)=ht(K)=2$ and $w(K')>w(K)$.
\end{remark}

\begin{figure}[htbp]
	\begin{center}
	\includegraphics[trim=0mm 0mm 0mm 0mm, width=.6\linewidth]{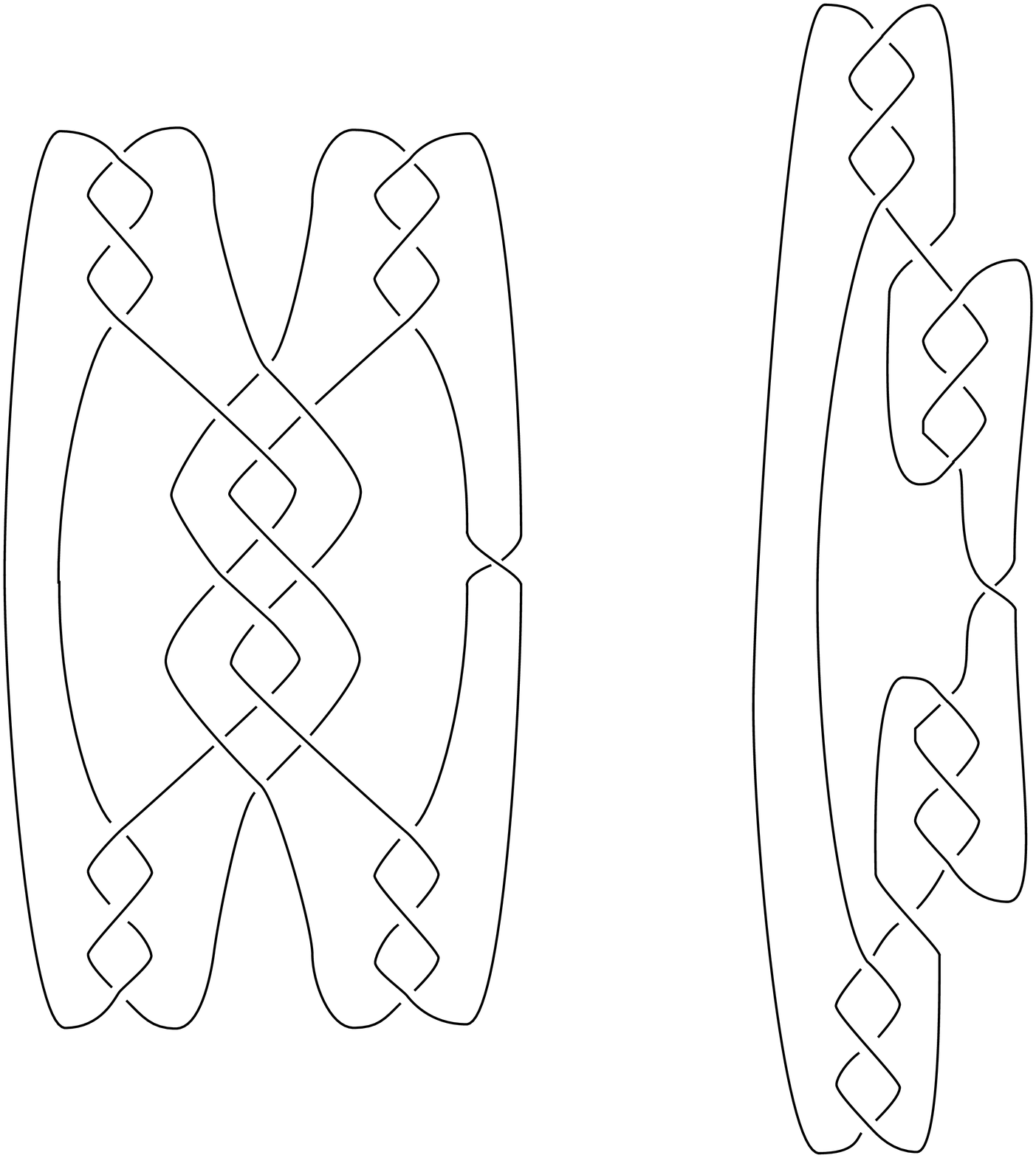}
	\end{center}
	\caption{A counterexample for Theorem \ref{htineq} without $w(K)=w(K')$}
	\label{counterexample3}
\end{figure}

In \cite{BT1}, Blair and Tomova construct an infinite collection of ambient isotopy classes of knots $\mathcal{K}$ from the schematic depicted in the left-hand side of Figure \ref{counterexample} by inserting suitable braids $B_1,...,B_4$ into the boxes shown. By Theorems 12.1 and 12.2 of \cite{BT1}, for all $K\in \mathcal{K}$, $w(K)=134$ and any thin position for $K$ has exactly three thick levels of width $10$ and exactly two thin levels of width $4$. Hence, $ht_{\min}(K)=ht(K)=3$ for all $K\in \mathcal{K}$. Note that if we consider the height function to be increasing from the bottom to the top of Figure \ref{counterexample}, then, for suitable choices of $B_1,...,B_4$, the left-hand side of the figure depicts a thin position for any knot in $\mathcal{K}$.

\begin{figure}[htbp]
	\begin{center}
	\begin{tabular}{cc}
	\includegraphics[trim=0mm 0mm 0mm 0mm, width=.35\linewidth]{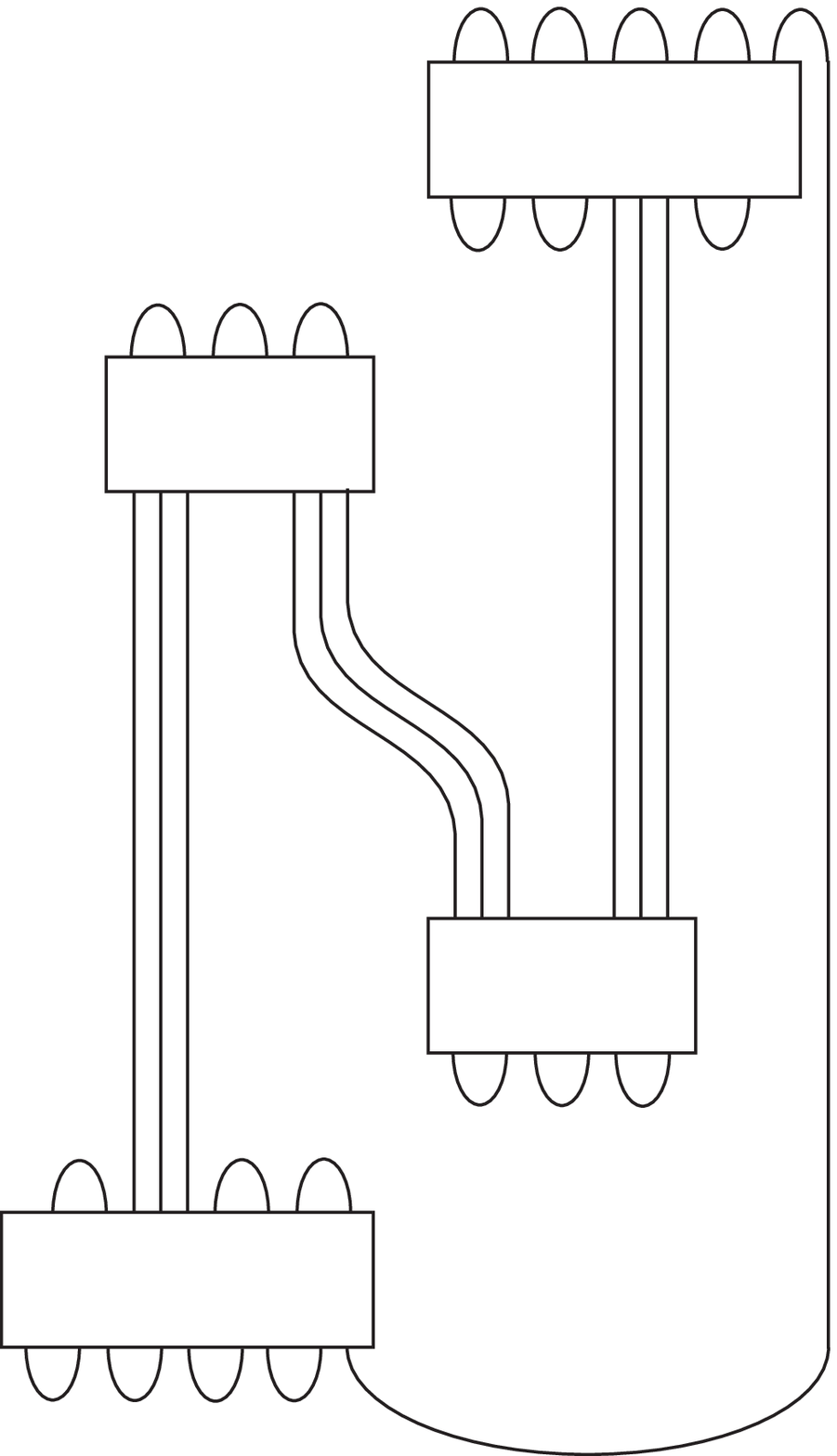}&
	\includegraphics[trim=0mm 0mm 0mm 0mm, width=.35\linewidth]{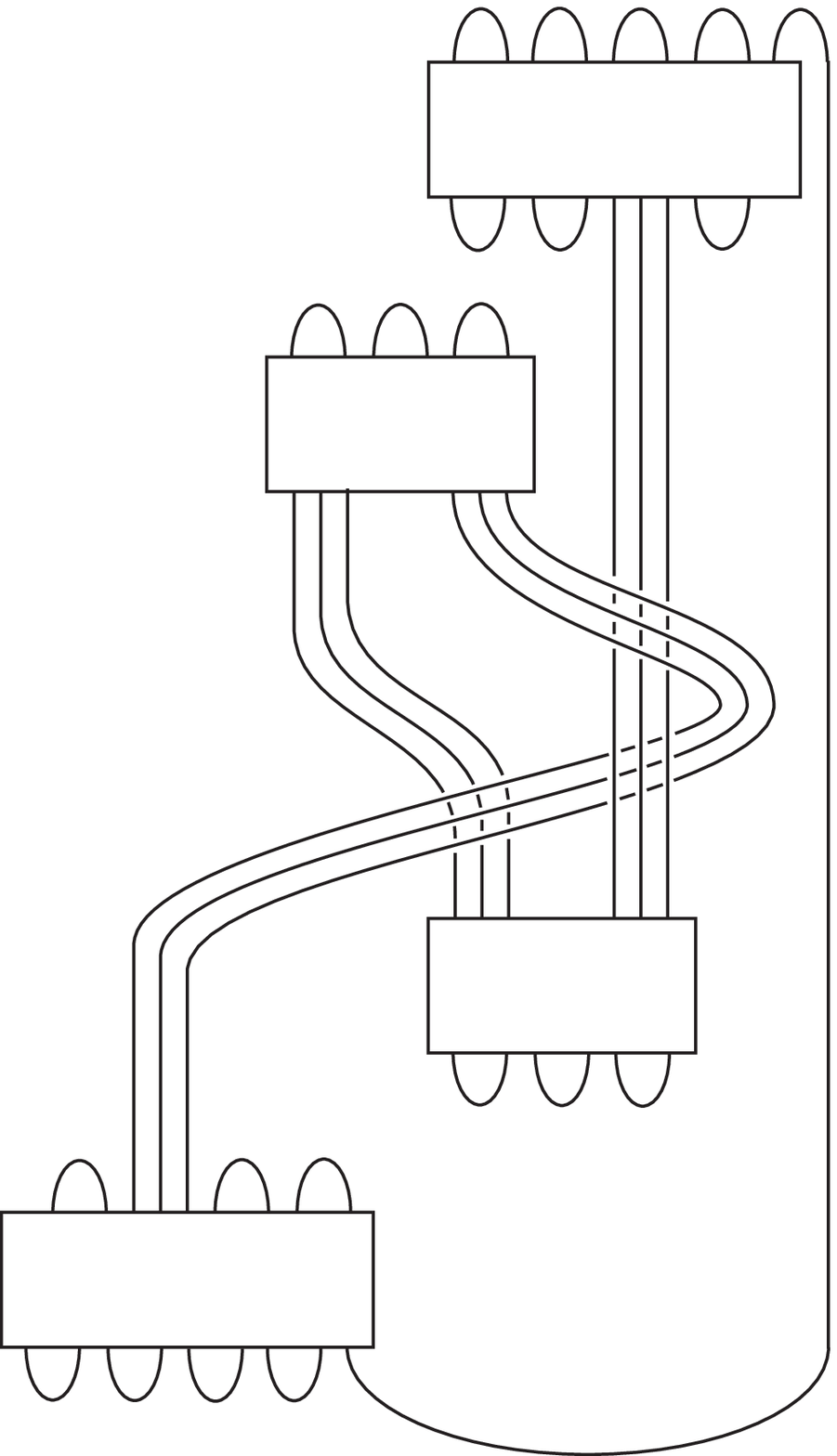}\\
	$\gamma\in K$ & \\
	\end{tabular}
	\end{center}
	\caption{$K\in\mathcal{K}$ and $K\# K_2$}
	\label{counterexample}
\end{figure}

\begin{proof}[Proof of Theorem \ref{main}]
By \cite{BT1}, $w(K)=134$ and Figure \ref{counterexample} gives a thin position for $K$. Figure \ref{counterexample} demonstrates that $w(K_2 \# K)\leq w(K)=134$. By Corollary 6.4 of \cite{SchSch}, $w(K_2 \# K)\geq w(K)$ and, therefore, $w(K_2 \# K)= w(K)$. Hence, Figure \ref{counterexample} demonstrates a thin position or $K_2 \# K$ with three thick levels. In particular, $ht(K_2 \# K)\geq 3$ and $ht_{\min}(K_2 \# K)\leq 3$.

If we apply Theorem \ref{htineq} to the embeddings of $K_2\# K$ and $K$ depicted in Figure \ref{counterexample} where $f(H)$ is the knotted ``swallow-follow'' torus that contains $K_2\# K$, swallows $K$ and follows $K_2$, then, since $w(K_2 \# K)= w(K)$, $ht(K_2 \# K)\leq ht(K)=3$ and $ht_{\min}(K_2 \# K)\geq ht_{\min}(K)=3$. Hence, $ht(K_2 \# K)=3$ and $ht_{\min}(K_2 \# K)=3$.
\end{proof}

\subsection{Proof of Theorem \ref{rt} and \ref{lrt}}

\begin{proof}[Proof of Theorem \ref{rt}]
Firstly, if $K$ is the trivial knot, then we have $r(K)=1$ and $trunk(K)=2$, and hence the inequality of Theorem \ref{rt} holds.

Next, we will show that for a non-trivial knot $K$, a height function $h:S^3\to \mathbb{R}$ and a closed surface $F$ containing $K$,
\[
r(F,K) \le \frac{\max_{t\in \Bbb{R}} |h^{-1}(t)\cap K|}{2}.
\]
By taking maximal of the left-hand side and minimal of the right-hand side, we have
\[
\max_{F\in\mathcal{F}} r(F,K) \le \frac{\min_{\gamma \in K} \max_{t\in \Bbb{R}} |h^{-1}(t)\cap \gamma|}{2}.
\]
Thus,
\[
\displaystyle r(K)\le \frac{trunk(K)}{2}.
\]

By perturbing $F$ relative to $K$, we may assume that any critical point of $F$ is not on $K$ and $F$ is also in a Morse position with respect to $h$.
Since the genus of $F$ is greater than $0$, there exists a regular value $t\in \mathbb{R}$ for $F$ such that $h^{-1}(t)\cap F$ contains at least two essential loops in $F$.
Take two loops $l_1,\ l_2$ of $h^{-1}(t)\cap F$ which are essential in $F$ and innermost in $h^{-1}(t)$.
Let $D_1,\ D_2$ be mutually disjoint disks in $h^{-1}(t)$ which are bounded by $l_1,\ l_2$ respectively.
Then we have
\[
|\partial D_1\cap K| + |\partial D_2\cap K| \le \max_{t\in \Bbb{R}} |h^{-1}(t)\cap K|.
\]
Without loss of generality, we may assume that
\[
\displaystyle |\partial D_1\cap K|\le \frac{\max_{t\in \Bbb{R}} |h^{-1}(t)\cap K|}{2}.
\]
By cutting and pasting $D_1$ if necessary, we may assume that $D_1\cap F=\partial D_1$.
Thus, $D_1$ is a compressing disk for $F$ and we have
\[
r(F,K)\le |\partial D_1\cap K|.
\]
\end{proof}

Let $K$ be a knot admitting an essential tangle decomposition $(S^3,K)=(B_1,T_1)\cup_S (B_2,T_2)$, where $S$ is a tangle decomposing sphere.
In the following, we show
\[
\displaystyle r(K)\le \frac{\min\{ trunk(B_1,T_1), trunk(B_2,T_2) \}}{2}.
\]

\begin{proof}[Proof of Theorem \ref{lrt}]
Let $F$ be a closed surface containing $K$. Let $h_i:B_i\to \mathbb{R}$ be a standard Morse function for $i=1,2$.
It is suffice to show that
\[
r(F,K) \le \frac{\max_{t\in \Bbb{R}} |h_i^{-1}(t)\cap T_i|}{2},
\]
for $i=1,2$.

We remark that $K$ is non-trivial since $K$ admits an essential tangle decomposition.
Hence the genus of $F$ is greater than $0$ and we may assume that $2\le r(F,K)$.
Since $F$ and $S$ are essential in the exterior of $K$, we may assume that each loop of $F\cap S$ is essential in $F$. If $F\cap S$ consists of a single essential loop, then both $F\cap B_1$ and $F\cap B_2$ are surfaces with strictly positive genus.
Then, there exists a regular value $t_i\in \mathbb{R}$ for $F$ such that $h_i^{-1}(t_i)\cap F$ contains at least two essential loops in $F$ for $i=1,2$. Otherwise,  $F\cap S$ consists of at least two essential loops and $h_i^{-1}(0)\cap F$ contains at least two essential loops in $F$ for $i=1,2$.
Similarly to Proof of Theorem \ref{rt}, we obtain a compressing disk $D$ for $F$ in $B_i$ such that
\[
\displaystyle r(F,K)\le |\partial D\cap T_i| \le \frac{\max_{t\in \Bbb{R}} |h^{-1}(t)\cap T_i|}{2}.
\]
\end{proof}

\section{Remarks}

\subsection{Several versions of thin position}
In the previous part of this paper, we considered Gabai's thin position $TP(K)$ (\cite{Gabai1}), that is, the set of all position $\gamma$ minimizing the width $w(\gamma)=\sum w(h^{-1}(r_{i}))$ for chosen regular values $r_{1},r_{2},...,r_{n}$.
Then we have already established the following.

\begin{enumerate}
\item There exists a candidate knot $K$ in Figure \ref{counterexample4} such that $ht(K)>ht_{\min}(K)$.
\item There exist two knots $K$ and $K'$ in Theorem \ref{main} such that $ht(K\# K')< ht(K)+ht(K')$.
\item There exists a candidate knot $K=K_{4,1,3,3}$ in \cite{DZ} such that $trunk(K)$ cannot be obtained in $TP(K)$.
\item Every thinnest level sphere for $\gamma\in TP(K)$ is incompressible in the complement of a knot \cite{W2004}.
\item There exists a knot $K$ in \cite{BZ} such that $\gamma\in TP(K)$ has a compressible thin level sphere.
\end{enumerate}

Next, let $MCP(K)$ be the set of all Morse positions of $K$ which have minimal critical points among all Morse positions.
We say that a knot belonging to $MCP(K)$ is in a {\em minimal critical position}.
Similarly, we can define the {\em MCP-height} and the {\em min-MCP-height} of $K$ as
\[
\displaystyle ht^{MCP}(K)=\max_{\gamma\in MCP(K)} ht(\gamma).
\]
\[
\displaystyle ht^{MCP}_{\min}(K)=\min_{\gamma\in MCP(K)} ht(\gamma).
\]
Note that $ht^{MCP}_{\min}(K)=1$ for any knot $K$.
Then we have the following.

\begin{enumerate}
\item There exists a knot $K$ in Figure \ref{stack} such that $ht^{MCP}(K)>ht^{MCP}_{\min}(K)$.
\item There exist two knots $K$ and $K'$ in Figure \ref{stack} such that $ht^{MCP}(K\# K')> ht^{MCP}(K)+ht^{MCP}(K')$.
\item There exists a candidate knot $K=K_{4,1,3,3}$ in \cite{DZ} such that $trunk(K)$ cannot be obtained in $MCP(K)$.
\item There exists a knot $K$ in Figure \ref{stack} such that $\gamma\in MCP(K)$ has a compressible thinnest level sphere.
\item There exists a knot $K$ in Figure \ref{stack} such that $\gamma\in MCP(K)$ has a compressible thin level sphere.
\end{enumerate}

\begin{figure}[htbp]
	\begin{center}
	\includegraphics[trim=0mm 0mm 0mm 0mm, width=.4\linewidth]{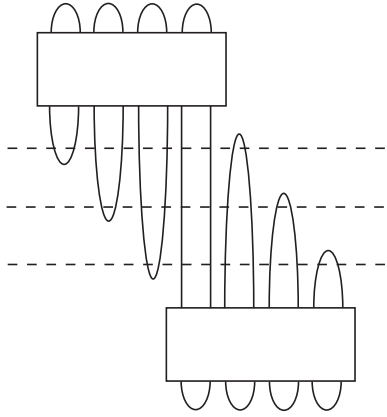}
	\end{center}
	\caption{A minimal critical position of $K\# K'$}
	\label{stack}
\end{figure}

Note that a knot $K$ in Figure \ref{stack} is not in a thin position, thus we have $\gamma\in MCP(K)\backslash TP(K)$.
Moreover, we remark that there exists a knot $K=K_\alpha$ in \cite{BT1} such that $TP(K)\cap MCP(K)=\emptyset$.

Finally, we define the {\em ordered} thin position $OTP(K)$, that is, a Morse position $\gamma$ of $K$ which minimizes the lexicographical order of monotonically non-increasing ordered set $\{w_i\}$, where $w_i$ $(i=1,\ldots,ht(\gamma))$ is the number of points of intersection between thick level spheres and $\gamma$.

For example, the embedding $\gamma\in K\# K'$ in Figure \ref{stack} has the complexity $\{ 10,10,10\}$, 
but it can be reduced to $\{8,8\}$ and we obtain an embedding $\gamma'\in OTP(K\# K')$.

Similarly, we define the {\em OTP-height} of $K$ as
\[
\displaystyle ht^{OTP}(K)=\max_{\gamma\in OTP(K)} ht(\gamma).
\]
\[
\displaystyle ht^{OTP}_{\min}(K)=\min_{\gamma\in OTP(K)} ht(\gamma).
\]
Then we have the following.

\begin{enumerate}
\item For any knot $K$, $ht^{OTP}(K)=ht^{OTP}_{\min}(K)$.
\item There exist a candidate knot $K_{4,1,3,3}$ in \cite{DZ} and a two-bridge knot $K_2$, $ht^{OTP}(K_{4,1,3,3}\# K_2)< ht^{OTP}(K_{4,1,3,3})+ht^{OTP}(K_2)$ as in Theorem \ref{main}.
\item $trunk(K)$ can be obtained in $OTP(K)$, as the first term of the monotonically non-increasing ordered set $\{w_i\}$.
\item Every thinnest level sphere for $\gamma\in OTP(K)$ is incompressible in the complement of a knot (by a similar argument to \cite{W2004}).
\item There exists a candidate embedding $\gamma\in K$ in \cite{BZ} such that $\gamma \in OTP(K)$ and $\gamma$ has a compressible thin level sphere.
\end{enumerate}

In Figure \ref{venn}, we summarize a relation on several versions of thin position.
For each region, we give an example of an embedding in the corresponding subset of Morse embeddings. Each of these examples is conjectural with the exception of the 2-bridge embedding and the embedding from Figure \ref{stack}, which can easily be verified.
Potential examples of embeddings $k_{2,1,3,7}$, $k'_{2,1,3,7}$, $k_{4,1,3,3}$ and $k'_{4,1,3,3}$ are referred from \cite{DZ}.
We have a potential example $\gamma'\in (TP(K)\cap OTP(K)) \backslash MCP(K)$ from Figure \ref{counterexample4}.

\begin{figure}[htbp]
	\begin{center}
	\includegraphics[trim=0mm 0mm 0mm 0mm, width=.45\linewidth]{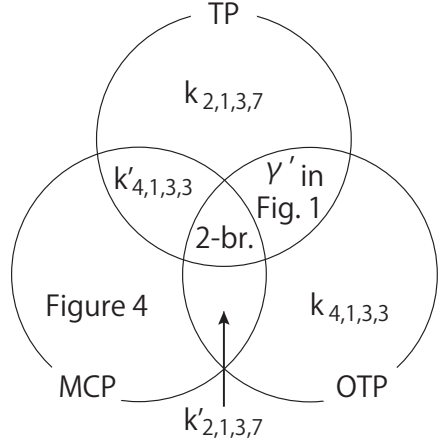}
	\end{center}
	\caption{Venn diagram for $TP$, $MCP$ and $OTP$}
	\label{venn}
\end{figure}

\subsection{Trunk-to-height ratio}
We propose that the definition of a proportion of a knot, like as a person's waist-to-height ratio, BMI and so on.
We define the {\em proportion} of a knot $K$ as
\[
\displaystyle pro(K)=\min_{\gamma\in TP(K)} \frac{trunk(\gamma)}{ht(\gamma)}\frac{1}{2\beta(\gamma)}.
\]
${trunk(\gamma)}/{ht(\gamma)}$ measures the slimness of the knot, and it is normalized by dividing by $2\beta(\gamma)$.
Thus we have $0<pro(K)\le 1$ for any knot $K$, and $pro(K)=1$ if and only if any thin position of $K$ is a bridge position.
It holds by \cite{T1997} that if $mp(K)=0$, then $pro(K)=1$.
We remark that $trunk(K)$ and $\beta(K)$ may not be obtained in $TP(K)$ as noted in the previous subsection.

Similarly, we can also define the {\em MCP-proportion} and {\em OTP-proportion} of $K$ as
\[
\displaystyle pro^{MCP}(K)=\min_{\gamma\in MCP(K)} \frac{trunk(\gamma)}{ht(\gamma)}\frac{1}{2\beta(K)},%=\frac{1}{2\beta(K)}\min_{\gamma\in MCP(K)} \frac{trunk(\gamma)}{ht(\gamma)}
\]
\[
\displaystyle pro^{OTP}(K)=\min_{\gamma\in OTP(K)} \frac{trunk(K)}{ht(\gamma)}\frac{1}{2\beta(\gamma)}.%=\frac{trunk(K)}{2}\min_{\gamma\in OTP(K)} \frac{1}{ht(\gamma)}\frac{1}{\beta(\gamma)}
\]

We can also consider another version by replacing ${trunk(\gamma)}/{ht(\gamma)}$ with the {\em average trunk} $at(\gamma)$, that is, the average of the intersection number of all thick level spheres and $\gamma$.

We expect that high distance knots have high proportions.
For example, if a knot $K$ has a bridge sphere of distance greater than or equal to $2\beta(K)$, then every thin position is a bridge position by using results of \cite{T1997} and \cite{BS2005}, and hence $pro(K)=1$.

\begin{example}\label{proexa}
For a knot $K$ given in Figure \ref{counterexample4}, we conjecture that there exists an embedding $\gamma\in (TP(K)\cap MCP(K))\backslash OTP(K)$ and that 
\[
pro^{MCP}(K)=\frac{22}{1}\frac{1}{2\times 11}=1.
\]
We also conjecture that there exists an embedding $\gamma'\in (TP(K)\cap OTP(K))\backslash MCP(K)$ and that
\[
pro(K)=pro^{OTP}(K)=\frac{18}{2}\frac{1}{2\times 13}=0.346\ldots .
\]
\end{example}

\begin{example}\label{proexa2}
For a knot $K\in \mathcal{K}$ given in Theorem \ref{main}, by \cite{BT1}, there exists $\gamma \in K$ in Figure \ref{counterexample} such that $\gamma\in TP(K)\backslash MCP(K)$ and $\gamma$ is not in $OTP(K)$ since there exists an embedding of $K$ with two thick levels of width $10$ and two thick levels of width $8$.
Since it is shown in \cite{BT1} that every $\gamma\in TP(K)$ has exactly three thick levels, each of width $10$, we have 
\[
pro(K)=\frac{10}{3}\frac{1}{2\times 11}=0.1515\ldots .
\]
Moreover, by \cite{BT1}, there is $\gamma'\in MCP(K)\backslash (TP(K)\cup OTP(K))$ in Figure \ref{mcp}, and we have
\[
pro^{MCP}(K)=\frac{12}{2}\frac{1}{2\times 10}=0.3.
\]
We conjecture that 
\[
pro^{OTP}(K)=\frac{10}{4}\frac{1}{2\times 12}=0.1388\ldots .
\]
\begin{figure}[htbp]
	\begin{center}
	\includegraphics[trim=0mm 0mm 0mm 0mm, width=.4\linewidth]{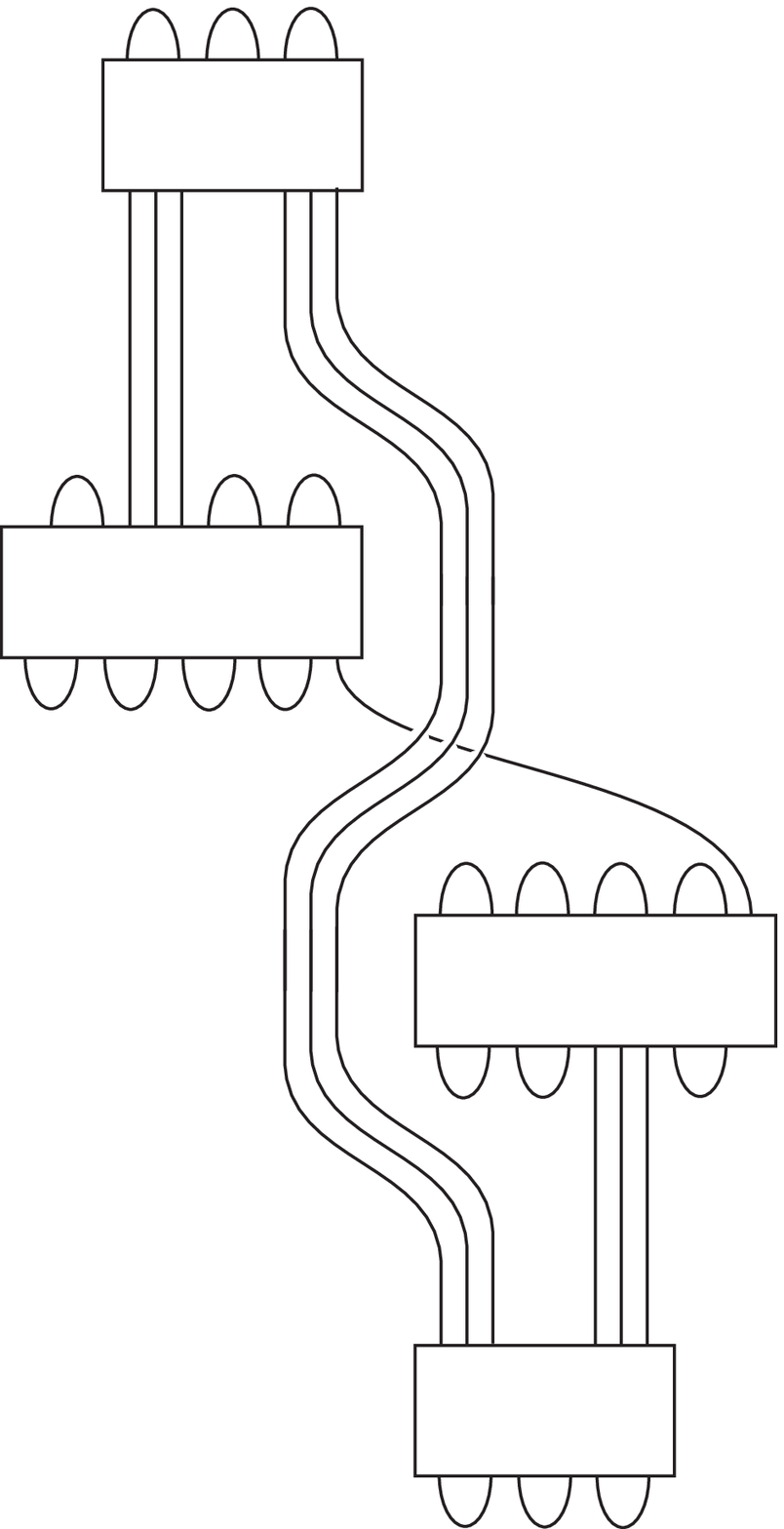}
	\end{center}
	\caption{$\gamma'\in MCP(K)\backslash (TP(K)\cup OTP(K))$, where $K\in \mathcal{K}$}
	\label{mcp}
\end{figure}
\end{example}

By Theorem \ref{rt}, we obtain the following.

\begin{corollary}
For any knot $K$,
\[
\displaystyle \frac{r(K)}{ht(K)\beta^{TP}_{max}(K)}\le pro(K),
\]
where $\beta^{TP}_{max}(K)$ denotes the {maximal} number of maximal points of $\gamma\in TP(K)$.
\end{corollary}

\subsection{Waist and representativity}
Theorem \ref{rt} is compared with the inequality between the waist and trunk of knots.
We define the {\em waist} of a knot $K$ as
\[
waist(K)=\max_{F\in\mathcal{F}} \min_{D\in\mathcal{D}_F} |D\cap K|,
\]
where $\mathcal{F}$ denotes the set of all closed surfaces in $S^3-K$, and $\mathcal{D}_F$ denotes the set of all compressing disks for $F$ in $S^3$ (\cite{O2}).
Then, we have $waist(K)=0$ for the trivial knot $K$ since any closed surface in $S^3-K$ is compressible, and by considering the peripheral torus $\partial N(K)$, $waist(K)\ge 1$ for non-trivial knots.
It is known that $waist(K)=1$ for $3$-braid knots (\cite{LP1985}), alternating knots (\cite{M1984}), almost alternating knots (\cite{A1992}), Montesinos knots (\cite{O1984}), toroidally alternating knots (\cite{A1994}), algebraically alternating knots (\cite{MO2010}), and that $waist(K) = p \cdot waist(J)$ for inconsistent cable knots with index $p$, where $J$ is a companion knot for $K$ (\cite{AKT}).

\begin{theorem}[{\cite[Theorem 1.9]{O2}}]\label{wt}
For any knot $K$, we have
\[
\displaystyle waist(K)\le \frac{trunk(K)}{3}.
\]
\end{theorem}

Theorem \ref{rt} and \ref{wt} bear a close resemblance to each other.
We expected in \cite[Problem 26]{O3} that $waist(K)\le r(K)$ for any knot $K$.
For example, any alternating knots satisfy this inequality since $waist(K)=1$ (\cite{M1984}) and $r(K)=2$ (\cite{K}).
However, it does not hold for composite knots in general.
The waist behaves as expected under taking connected sums, that is, $waist (K_1\# K_2) = \max\{waist (K_1),waist (K_2)\}$ (\cite[Proposition 1.2]{O2}).
On the other hand, we have $r(K_1\# K_2)=2$ whenever $K_1$ and $K_2$ are non-trivial.
This shows that the representativity of knots behaves dissimilarly to other geometric knot invariants.

\subsection{Representativity and non-orientable spanning surfaces}

Aumann proved that any alternating knot bounds an essential non-orientable spanning surface (\cite{A}).
Indeed, he showed that both checkerboard surfaces are essential.
Recently, Kindred proved in \cite{K} that $r(K)=2$ for any non-trivial alternating knot $K$, which confirmed Conjecture 4 in \cite{O3}.
From these results, we expect that 
if a knot $K$ bounds an essential non-orientable spanning surface, then $r(K)=2$.
However, we have the next proposition.

\begin{theorem}
For any integer $n\ge 2$, there exists a knot with $r(K)\ge n$ which bounds an essential once punctured Klein bottle.
\end{theorem}

\begin{proof}
Let $V\cup_F W$ be a genus two Heegaard splitting of $S^3$.
Take a loop $C$ on $\partial V$ as shown in Figure \ref{counterexample2}.
Note that $C$ bounds a M\"{o}bius band $M$ properly embedded in $V$ which is formed by a non-separating disk and a band.
Let $A$ be a loop obtained from a train track $T$ on $\partial V$ as shown in Figure \ref{counterexample2}, where $m,\ n\ge 1$.
By adding a band $B$ along $A$ to $M$, we obtain a once punctured Klein bottle $F=M\cup B$ properly embedded in $V$ and a knot $K=\partial F$.

It is easy to see that $K$ is {\em $\min\{2m,2n+2\}$-seamed} with respect to a complete set of essential disks $\{D_1,D_2,D_3\}$ in $V$, that is, $K$ has been isotoped to intersect $\bigcup \partial D_i$ minimally and for each pair of pants $P$ obtained from $\partial V$ by cutting along $\bigcup \partial D_i$, and for each pair of two boundary components of $P$, there exist at least $\min\{2m,2n+2\}$ arcs of intersection in $K\cap P$ that connect that pair of boundary components.
The following lemma can be proved by an elementary cut and paste argument.

\begin{lemma}
If $K\subset \partial V$ is $k$-seamed with respect to a complete set of meridian disks $\{D_i\}$ in $V$ and $\Delta$ be an essential disk in $V$, then $|K\cap \partial \Delta|\ge 2k$.
\end{lemma}

By this lemma, for any compressing disk $\Delta$ for $F$ in $V$, $\partial \Delta$ intersects $K$ at least $2\min\{2m,2n+2\}$ points.

Finally, to obtain a knot $K$ with $r(K)\ge 2\min\{2m,2n+2\}$, we re-embed $V$ in $S^3$ so that $S^3-int V$ is boundary-irreducible.
Then there exists no compressing disk for $F$ in $S^3-int V$, and we have $r(F,K)\ge 2\min\{2m,2n+2\}$.
\end{proof}

\begin{figure}[htbp]
	\begin{center}
	\includegraphics[trim=0mm 0mm 0mm 0mm, width=.8\linewidth]{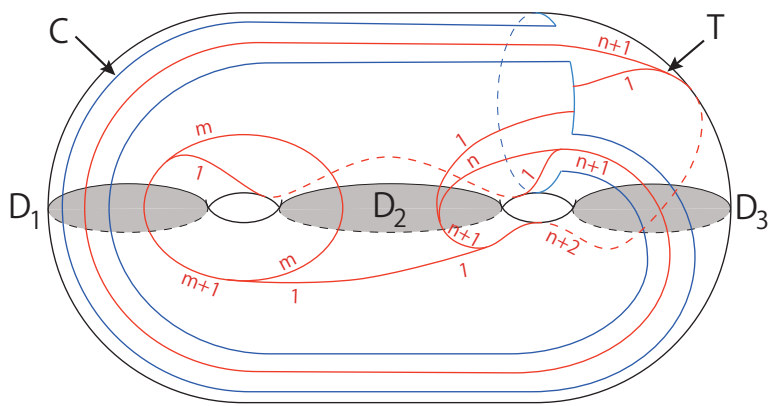}
	\end{center}
	\caption{Constructing a knot $K$ with $r(F,K)\ge n$ which bounds an essential once punctured Klein bottle}
	\label{counterexample2}
\end{figure}

\begin{remark}
We remark that there exists a knot which does not bound an essential non-orientable spanning surface (\cite{D}).
\end{remark}

\bigskip

\noindent{\bf Acknowledgements.}

We would like to thank to Koya Shimokawa for suggesting to consider the Scharlemann--Thompson type thin position.

\bibliographystyle{amsplain}

\begin{thebibliography}{10}

\bibitem{A1992} C. Adams, J. Brock, J. Bugbee, T. Comar, K. Faigin, A. Huston, A. Joseph, D. Pesikoff, {\em Almost alternating links}, Topology Appl. {\bf 46} (1992) 151--165.

\bibitem{A1994} C. Adams, {\em Toroidally alternating knots and links}, Topology {\bf 33} (1994) 353--369.

\bibitem{AKT} R. Aranda, S. Kim, M. Tomova, {\em Representativity and waist of cable knots}, arXiv:1704.08414.

\bibitem{A} R. J. Aumann, {\em Asphericity of alternating knots}, Ann. of Math. {\bf 64} (1956), 374--392.

\bibitem{B} D. Bachman, {\em Non-parallel essential surfaces in knot complements}, arXiv:math/0302272.

\bibitem{BS2005} D. Bachman, S. Schleimer, {\em Distance and bridge position}, Pacific J. Math. {\bf 219} (2005) 221--235.

\bibitem{BThesis} R. Blair. \newblock {\em Bridge number and Conway products}. \newblock (2010). Thesis (Ph.D.)-University of California, Santa Barbara.

\bibitem{BT1} R. Blair, M. Tomova. \newblock {\em Width is not additive}. \newblock {Geom. Topol.}, 17:93--156, 2013.

\bibitem{BZ} R. Blair, A. Zupan, {\em Knots with compressible thin levels}, Alg. \& Geom. Top. {\bf 15} (2015), 1691--1715.

\bibitem{DZ} D. Davies, A. Zupan, {\em Natural properties of the trunk of a knot}, arXiv:1608.00019, to appear in J. Knot Theory and its Ramifications.

\bibitem{D} N. M. Dunfield, {\em A knot without a nonorientable essential spanning surface}, Illinois J. Math. {\bf 60} (2016), 179--184.

\bibitem{Gabai1} D. Gabai. \newblock {\em Foliations and the topology of {$3$}-manifolds. {III}}. \newblock {J. Differential Geom.}, 26:479--536, 1987.

\bibitem{HK1997} D. J. Heath, T. Kobayashi, {\em Essential tangle decomposition from thin position of a link}, Pacific J. Math. {\bf 179} (1997) 101--117.

\bibitem{IPSSVAS} K. Ishihara, M. Pouokam, A. Suzuki, R. Scharein, M. Vazquez, J. Arsuaga, K. Shimokawa, {\em Bounds for the minimum step number of knots confined to tubes in the simple cubic lattice}, J. Phys. A: Math. Theor. {\bf 50} (2017),  215601.

\bibitem{JT} J. Johnson, M. Tomova, {\em Flipping bridge surfaces and bounds on stable bridge number}, Alg. Geom. Topol. {\bf 11} (2011), 1987--2005.

\bibitem{K} T. Kindred, {\em Alternating links have representativity 2}, arXiv:1703.03393.

\bibitem{K2} D. A. Krebes, {\em An obstruction to embedding 4-tangles in links}, J. Knot Theory and its Ramifications {\bf 8} (1999), 321--352.

\bibitem{LP1985} M. T. Lozano, J. H. Przytycki, {\em Incompressible surfaces in the exterior of a closed 3-braid I, surfaces with horizontal boundary components}, Math. Proc. Camb. Phil. Soc. {\bf 98} (1985) 275--299.

\bibitem{M1984} W. Menasco, {\em Closed incompressible surfaces in alternating knot and link complements}, Topology {\bf 23} (1984) 37--44.

\bibitem{Mori1} K. Morimoto. \newblock {\em There are knots whose tunnel numbers go down under connected sum}. \newblock {Proc. Amer. Math. Soc.} {\bf 123} (1995), 3527--3532.

\bibitem{O1984} U. Oertel, {\em Closed incompressible surfaces in complements of star links}, Pacific J. Math. {\bf 111} (1984) 209--230.

\bibitem{Ozawa2} M. Ozawa. \newblock {\em On uniqueness of essential tangle decompositions of knots
with free tangle decompositions}. \newblock {Proc. Appl. Math.}, Workshop 8, ed.
G. T. Jin and K. H. Ko, 227--232, 1998.

\bibitem{O1}M. Ozawa, {\em Non-triviality of generalized alternating knots}, J. Knot Theory and its Ramifications {\bf 15} (2006) 351--360.

\bibitem{MO2010} M. Ozawa, {\em Rational structure on algebraic tangles and closed incompressible surfaces in the complements of algebraically alternating knots and links}, Topology and its Applications 157 (2010) 1937--1948.

\bibitem{O2} M. Ozawa, {\em Waist and trunk of knots}, Geometriae Dedicata {\bf 149} (2010) 85--94.

\bibitem{O3} M. Ozawa, {\em Bridge position and the representativity of spatial graphs}, Topology and its Appl. {\bf 159} (2012) 936--947.

\bibitem{O4} M. Ozawa, {\em The representativity of pretzel knots}, J. Knot Theory and its Ramifications {\bf 21}, 1250016 (2012).

\bibitem{Ozawa1} M. Ozawa, \newblock {\em Knots and surfaces}. \newblock {S\=ugaku}, 67:403--423, 2015. (An English translation is available at arXiv:1603.09039.)

\bibitem{P} J. Pardon, {\em On the distortion of knots on embedded surfaces}, Ann. of Math. (2) {\bf 174} (2011) 637-646.

\bibitem {RS} Y. Rieck, E. Sedgwick. \newblock {\em Thin position for a connected sum of small knots},
    \newblock {Algebr. Geom. Topol.}, 2:297--309 (electronic), 2002.

\bibitem{RV} N. Robertson, R. Vitray, {\em Representativity of Surface Embeddings}, Proceedings of the ”Paths, Flows and VLSI-Layout” conference, Universitat, Bonn. 1990.

\bibitem{R} D. Ruberman, {\em Embedding tangles in links}, J. Knot Theory and its Ramifications {\bf 9} (2000), 523--530.

\bibitem{SchSch} M. Scharlemann, J. Schultens. \newblock {\em 3-manifolds with planar presentations and
    the width of satellite knots}. \newblock {Trans. Amer. Math. Soc.}, 358:3781--3805, 2006.

\bibitem{ST} M. Scharlemann, A. Thompson, {\em Thin position for 3-manifolds}, AMS Contemporary Math. {\bf 164} (1994), 231--238.
 
\bibitem{SchTh} M. Scharlemann, A. Thompson. \newblock {\em On the additivity of knot width.} \newblock
    {Proceedings of the Casson Fest, Geometry and Topology Monographs} 7 (2004) 135-144, 12(5):683--708,
    2003.

\bibitem{S} H. Schubert, {\em \"{U}ber eine numerische Knoteninvariante}, Math. Z. {\bf 61} (1954) 245--288.

\bibitem{T} K. Taniyama, {\em Multiplicity of a space over another space}, J. Math. Soc. {\bf 64} (2012),  823--849.

\bibitem{TT1} S. Taylor, M. Tomova. \newblock {\em Additive invariants for knots, links
and graphs in 3-manifolds}. \newblock {arXiv:1606.03408}, 2016.

\bibitem{T1997} A. Thompson, {\em Thin position and bridge number for knots in the 3-sphere}, Topology {\bf 36} (1997), 505--507.

\bibitem{W2004} Y.-Q. Wu, {\em Thin position and essential planar surfaces}, Proc. Amer. Math. Soc. {\bf 132} (2004) 3417--3421.


\end{thebibliography}

\end{document}